\documentclass[12pt,reqno]{article}
\usepackage{diagbox}
\usepackage{scalerel}[2016/12/29]
\usepackage{caption}
\usepackage[usenames]{color}
\usepackage{amssymb}
\usepackage{amsmath}
\usepackage{amsthm}
\usepackage{amsfonts}
\usepackage{amscd}
\usepackage{graphicx}
\usepackage[shortlabels]{enumitem}
\usepackage{verbatim}
\usepackage{esdiff}
\usepackage[colorlinks=true,
linkcolor=webgreen,
filecolor=webbrown,
citecolor=webgreen]{hyperref}

\definecolor{webgreen}{rgb}{0,.5,0}
\definecolor{webbrown}{rgb}{.6,0,0}

\usepackage{color}
\usepackage{fullpage}
\usepackage{float}

\usepackage{mathtools}
\usepackage{graphics}
\usepackage{latexsym}
\usepackage{epsf}
\usepackage{breakurl}

\begin{document}

\theoremstyle{plain}
\newtheorem{theorem}{Theorem}
\newtheorem{corollary}[theorem]{Corollary}
\newtheorem{lemma}[theorem]{Lemma}
\newtheorem{proposition}[theorem]{Proposition}

\theoremstyle{definition}
\newtheorem{definition}[theorem]{Definition}
\newtheorem{example}[theorem]{Example}
\newtheorem{conjecture}[theorem]{Conjecture}
\newtheorem{open}[theorem]{Open Problem}

\theoremstyle{remark}
\newtheorem{remark}{Remark}

\title{Asymptotic bounds for the number of closed and privileged words}

\author{Daniel Gabri\'{c} \\
School of Computer Science \\
University of Guelph \\
Guelph, ON  N1G 2W1 \\
Canada\\
\href{mailto:dgabric@uoguelph.ca}{\tt dgabric@uoguelph.ca} }
\date{}

\maketitle

\begin{abstract}
A word~$w$ has a border $u$ if $u$ is a non-empty proper prefix and suffix of $u$. A word~$w$ is said to be \emph{closed} if $w$ is of length at most $1$ or if $w$ has a border that occurs exactly twice in $w$. A word~$w$ is said to be \emph{privileged} if $w$ is of length at most $1$ or if $w$ has a privileged border that occurs exactly twice in $w$. Let $C_k(n)$ (resp.~$P_k(n)$) be the number of length-$n$ closed (resp. privileged) words over a $k$-letter alphabet. In this paper, we improve existing upper and lower bounds on $C_k(n)$ and $P_k(n)$. We completely resolve the asymptotic behaviour of $C_k(n)$. We also nearly completely resolve the asymptotic behaviour of $P_k(n)$ by giving a family of upper and lower bounds that are separated by a factor that grows arbitrarily slowly.
\end{abstract}

\section{Introduction}

%Bordered~\cite{Silberger:1971,Nielsen:1973}

Let $\Sigma_k$ denote the $k$-letter alphabet $\{0,1,\ldots, k-1\}$. Throughout this paper, we denote the length of a word~$w$ as $|w|$. A word~$u$ is said to be a \emph{factor} of a word~$w$ if $w=xuy$ for some words $x$, $y$. A word~$w$ has a \emph{border} $u$ if $u$ is a non-empty proper prefix and suffix of $w$. A word that has a border is said to be \emph{bordered}; otherwise, it is said to be \emph{unbordered}.  A word~$w$ is said to be \emph{closed} if $|w|\leq 1$ or if $w$ has a border that occurs exactly twice in $w$.  If $u$ is a border $w$ and $u$ occurs in $w$ exactly twice, then we say $w$ is \emph{closed by} $u$. It is easy to see that if a word~$w$ is closed by a word~$u$, then $u$ must be the largest border in $w$; otherwise $u$ would occur more than two times in $w$. A word~$w$ is said to be \emph{privileged} if $|w| \leq 1$ or if $w$ is closed by a privileged word.

\begin{example}\label{example:closedprivileged}
The English word {\tt entanglement} has the border {\tt ent} and only contains two occurrences of {\tt ent}. Thus, {\tt entanglement} is a closed word, closed by {\tt ent}. Since $|{\tt ent}|>1$ and {\tt ent} is unbordered and therefore not privileged, we have that {\tt entanglement} is not privileged.

The English word {\tt alfalfa} is closed by {\tt alfa}. Furthermore, {\tt alfa} is closed by {\tt a}. But $|{\tt a}|\leq 1$, so {\tt alfa} is privileged and therefore so is {\tt alfalfa}.

The only border of the English word {\tt eerie} is {\tt e} and {\tt e} appears $3$ times in the word. Thus, {\tt eerie} is neither closed nor privileged. 
\end{example}

Closed words were introduced relatively recently by Fici~\cite{Fici:2011} as a way to classify Trapezoidal and Sturmian words. However, there are multiple equivalent formulations of closed words that have been defined at different times. Closed words are equivalent to codewords in prefix-synchronized codes~\cite{Gilbert:1960,Guibas&Odlyzko:1978}.  Closed words are also equivalent to \emph{periodic-like} words~\cite{Carpi&deLuca:2001}. A \emph{period} of a word $w=w_1w_2\cdots w_n$ is an integer $p\leq n$ such that $w_i = w_{i+p}$ for all $1\leq i \leq n-p$. A length-$n$ word is said to be \emph{periodic} if it has a period of length $\leq n/2$. In applications that require the analysis of long words, like DNA sequence analysis, the smallest period is typically much larger than half the length of the word. Periodic-like words were introduced as a generalization of periodic words that preserve some desirable properties of periodic words.

Privileged words~\cite{Kellendonk&Lenz&Savinien:2013} were introduced as a technical tool related to a problem in dynamical systems and discrete geometry. They were originally defined as a generalization of \emph{rich} words by tweaking the definition of a \emph{complete first return}. A \emph{complete first return} to a word $u$ is a word that starts and ends with $u$, and contains only two occurrences of $u$. A \emph{palindrome} is a word that reads the same forwards and backwards. A word $w$ is said to be \emph{rich} if and only if every palindromic factor of $w$ is a complete first return to a shorter palindrome. Interestingly, rich words contain the maximum possible number of distinct palindromic factors. Privileged words were then defined as an iterated complete first return. A word is privileged if and only if it is a complete first return to a shorter privileged word. Single letters and the empty word are defined to be privileged in order to make this definition meaningful.

 Since their introduction, there has been much research into the properties of closed and privileged words~\cite{Badkobeh&Fici&Liptak:2015,Bucci:2013,DeLuca&Fici&Karhumaki&Lepisto&Zamboni:2013, Fici:2017, Jahannia:2022,Peltomaki:2013,Peltomaki:2015,  Schaeffer&Shallit:2016}. One problem that has received some interest lately~\cite{Forsyth2016, Nicholson2018, Rukavicka2020, Rukavicka:2022} is to find good upper and lower bounds for the number of closed and privileged words.

Let $C_k(n)$ denote the number of length-$n$ closed words over $\Sigma_k$. Let $C_k(n,t)$ denote the number of length-$n$ closed words over $\Sigma_k$ that are closed by a length-$t$ word. Let $P_k(n)$ denote the number of length-$n$ privileged words over $\Sigma_k$. Let $P_k(n,t)$ denote the number of length-$n$ privileged words over $\Sigma_k$ that are closed by a length-$t$ privileged word. See Tables~\ref{table:ClosedTable} and~\ref{table:PrivilegedTable} for sample values of $C_2(n,t)$ and $P_2(n,t)$ for small $n$, $t$. See sequences \href{https://oeis.org/A226452}{\underline{A226452}} and \href{https://oeis.org/A231208}{\underline{A231208}} in the \emph{On-Line Encyclopedia of Integer Sequences}~\cite{OEIS} for sample values of $C_2(n)$ and $P_2(n)$.

Every privileged word is a closed word, so any upper bound on $C_k(n)$ is also an upper bound on $P_k(n)$. Furthermore, any lower bound on $P_k(n)$ is also a lower bound on $C_k(n)$.

\begin{itemize}
\item Forsyth et al.~\cite{Forsyth2016} showed that $P_2(n) \geq 2^{n-5}/n^2$ for all $n\geq n_0$ for some $n_0>0$. 
\item Nicholson and Rampersad~\cite{Nicholson2018} improved and generalized this bound by showing that there are constants $c$ and $n_0$ such that $P_k(n) \geq c\frac{k^n}{n(\log_k(n))^2}$ for all $n \geq n_0$. 
\item Rukavicka~\cite{Rukavicka2020} showed that there is a constant $c$ such that $C_k(n) \leq c\ln n\frac{k^n}{\sqrt{n}}$ for all $n>1$. 
\item Rukavicka~\cite{Rukavicka:2022} also showed that for every $j\geq 3$, there exist constants $\alpha_j$ and $n_j$ such that $P_k(n) \leq \alpha_j\frac{k^n\sqrt{\ln n}}{\sqrt{n}}\ln^{\circ j}(n)\prod\limits_{i=2}^{j-1}\sqrt{\ln^{\circ i}(n)}$ length-$n$ privileged words for all $n\geq n_j$ where $\ln^{\circ 0}(n) = n$ and $\ln^{\circ j}(n) = \ln(\ln^{\circ j-1}(n))$.
\end{itemize}
The best upper and lower bounds for both $C_k(n)$ and $P_k(n)$ are widely separated, and can be much improved. In this paper, we improve the existing upper and lower bounds on $C_k(n)$ and $P_k(n)$. Let $\log_k^{\circ 0}(n)=n$ and $\log_k^{\circ j}(n) = \log_k(\log_k^{\circ j-1}(n))$ for $j\geq 1$. We prove the following two theorems. 
\begin{theorem}\label{theorem:mainC}
Let $k\geq 2$ be an integer. 
\begin{enumerate}[label=(\alph*)]
    \item There exist constants $N$ and $c$ such that $C_k(n)\geq c\frac{k^n}{n}$ for all $n>N$. \label{theorem:mainCLower}
    \item There exist constants $N'$ and $c'$ such that $C_k(n) \leq c'\frac{k^n}{n}$ for all $n>N'$. \label{theorem:mainCUpper}
\end{enumerate}
\end{theorem}

\begin{theorem}\label{theorem:mainP}
Let $k\geq 2$ be an integer. 
\begin{enumerate}[label=(\alph*)]
    \item For all $j\geq 0$ there exist constants $N_j$ and $c_j$ such that \[P_k(n)\geq c_j\frac{k^n}{n\log_k^{\circ j}(n)\prod_{i=1}^j\log_k^{\circ i}(n)}\] for all $n>N_j$. \label{theorem:mainPLower}
    \item For all $j\geq 0$ there exist constants $N_j'$ and $c_j'$ such that \[P_k(n) \leq c_j'\frac{k^n}{n\prod_{i=1}^j\log_k^{\circ i}(n)}\]
    for all $n>N_j'$.\label{theorem:mainPUpper}
\end{enumerate}
\end{theorem}

Before we proceed, we give a heuristic argument as to why $C_k(n)$ is in $\Theta(\frac{k^n}{n})$. Consider a ``random" length-$n$ word~$w$. Let $\ell = \log_k(n) + c$ where $c$ is a constant such that $\ell$ is a positive integer. There is a $\frac{1}{k^{\ell}}=\frac{1}{k^cn}$ chance that $w$ has a length-$\ell$ border. Suppose $w$ has a length-$\ell$ border. Now suppose we drop the first and last character of $w$ to get $w'$. If $w'$ were randomly chosen (which it is not), then we could use the linearity of expectation to get that the expected number of occurrences of $u$ in $w'$ is approximately $(n-1-\ell)k^{-\ell} \approx k^{-c}$. Thus, for $c$ large enough we have that $u$ does not occur in $w'$ with high probability, and so $w$ is closed. Therefore, there are approximately $k^{n-\ell} \in \Theta(\frac{k^n}{n})$ length-$n$ closed words.

\begin{table}[H]

\centering
\begin{tabular}{|c|cccccccccc|}
\hline
\backslashbox{$n$}{$t$}  & $1$ & $2$ & $3$ & $4$ & $5$ & $6$ & $7$ & $8$ & $9$ & $10$  \\
\hline
10 &  2 & 30 & 70 &  50 & 30    & 12  & 6  & 2  & 2  &   0 \\
%\hline
11 &  2 & 42 & 118 & 96 & 54   & 30  & 13  & 6  & 2  &  2  \\
%\hline
12 &  2 & 60 & 200 & 182 & 114  & 54  & 30  & 12  &  6 & 2    \\
%\hline
13 & 2 & 88 & 338 & 346 & 214   & 126  & 54  & 30   & 12  & 6   \\
%\hline
14 &  2 & 132 & 570 & 640 &  432  & 232  & 126  & 54  & 30  &   12 \\
%\hline
15 &  2 & 202 & 962 & 1192 & 828   & 474  & 240  & 126  & 54  & 30    \\
%\hline
16 &  2 & 314 & 1626  & 2220 & 1612    & 908  & 492  & 240  & 126  &  54 \\
%\hline
17 &  2 & 494 & 2754 & 4128 &  3112  & 1822  & 956  & 504  & 240  & 126    \\
%\hline
18 &  2  & 784  & 4676  & 7670  & 6024    & 3596  & 1934  & 982  & 504  & 240    \\
%\hline
19 &  2 & 1252 & 7960 & 14264 & 11636  & 7084  & 3828  & 1992  & 990  & 504   \\
%\hline
20 &  2 & 2008 & 13588 &  26524 & 22512  &  13928  & 7632  &  3946  & 2026  &  990  \\
\hline
\end{tabular}
\captionsetup{justification=centering}
\caption{Some values of $C_2(n,t)$ for $n$, $t$ where $10 \leq n \leq 20$ and $1\leq t \leq 10$.}
\label{table:ClosedTable}
\end{table}

\begin{table}[H]
\centering
\begin{tabular}{|c|cccccccccc|}
\hline
\backslashbox{$n$}{$t$}  & $1$ & $2$ & $3$ & $4$ & $5$ & $6$ & $7$ & $8$ & $9$ & $10$ \\
\hline
10  & 2 & 16 & 22 & 8  & 6   & 2  & 2  & 0   &  2  & 0    \\
%\hline
11  & 2  & 26 & 38 & 16 & 10   & 6  & 4  & 2  & 2  &   2  \\
%\hline
12  & 2 & 42 & 68 & 30 & 18   & 4  & 6  & 2  &  2 &   0   \\
%\hline
13  & 2 & 68  & 122 & 58 & 38  &  14 & 10  & 6  &  4 &   2  \\
%\hline
14  & 2 & 110 & 218 & 108 & 76  & 20  & 14  &  8 &  6  &  2  \\
%\hline
15  & 2 & 178 & 390 & 204 & 148   & 46  & 24  & 18  & 14  & 6   \\
%\hline
16  & 2 & 288 & 698 & 384 & 288  & 86  &  48 & 16  &  18  &  8  \\
%\hline
17  & 2 & 466 & 1250 & 724 & 556   & 178  & 92  & 36  &  32  & 26     \\
%\hline
18  & 2 & 754 & 2240 & 1364 & 1076    & 344  & 190  & 64  &  36  & 28    \\
%\hline
19 & 2 & 1220 & 4016 & 2572  & 2092   & 688  & 388  & 136  &  70  &  56   \\
%\hline
20 & 2 & 1974 & 7204 & 4850  & 4068   &  1342 & 772  &  268 &  138  &   52  \\
\hline
\end{tabular}
\captionsetup{justification=centering}
\caption{Some values of $P_2(n,t)$ for $n$, $t$ where $10 \leq n \leq 20$ and $1\leq t \leq 10$.}
\label{table:PrivilegedTable}
\end{table}

\begin{comment}
\begin{table}[H]
\centering
\begin{tabular}{|c|c|c|c|c|c|}
\hline
$n$ & $P_2(n)$  & $C_2(n)$ & $n$ & $P_2(n)$ & $C_2(n)$  \\
\hline
0 & 1 & 1 & 13 & 328 & 1220  \\
1 & 2 & 2 & 14 & 568 & 2240  \\
%\hline
2 & 2 & 2  & 15 & 1040 & 4132   \\
%\hline
3 & 4 & 4 & 16 & 1848 & 7646   \\
%\hline
4 & 4 & 6 & 17  & 3388 & 14244   \\
%\hline
5 & 8 & 12 & 18 & 6132 & 26644   \\
%\hline
6 & 8 & 20 & 19 & 11332 & 49984   \\
%\hline
7 & 16 & 36 & 20 & 20788 & 94132    \\
%\hline
8 & 20 & 62 & 21 & 38576 & 177788     \\
%\hline
9 & 40 & 116 & 22 & 71444 & 336756     \\
%\hline
10 & 60 & 204 & 23 & 133256 & 639720    \\
11 & 108 & 364 & 24 & 248676 & 1218228    \\
12 & 176 & 664 & 25 & 466264 & 2325048    \\
\hline
\end{tabular}
\captionsetup{justification=centering}
\caption{Some values of $P_2(n)$ and $C_2(n)$ for $n\leq 25$.}
\label{table:PrivilegedClosedTable}
\end{table}
\end{comment}

\section{Preliminary results}
In this section we give some necessary results and definitions in order to prove our main results. Also throughout this paper, we use $c$'s, $d$'s, and $N$'s to denote positive real constants (dependent on $k$).

Let $w$ be a length-$n$ word. Suppose $w$ is closed by a length-$t$ word~$u$. Since $u$ is also the largest border of $w$, it follows that $w$ cannot be closed by another word. This implies that \[C_k(n) = \sum_{i=1}^{n-1}C_k(n,t)\text{ and }P_k(n) = \sum_{i=1}^{n-1}P_k(n,t)\] for $n>1$.

Let $B_k(n,u)$ denote the number of length-$n$ words over $\Sigma_k$ that are closed by the word~$u$. Let $A_k(n,u)$ denote the number of length-$n$ words over $\Sigma_k$ that do not contain the word~$u$ as a factor.

%%%%%%%%%%%%%%%%%%%%%%%%%%%%%%%%%%%%%%%%%%%%%%%%%%%%%%%%%%%%%%%%%%%%%%%%%%%%%%%%%%%%%%%%%%%%%%%%%%%

The \emph{auto-correlation}~\cite{Guibas&Odlyzko:1978,Guibas&Odlyzko:1980, Guibas&Odlyzko:1981} of a length-$t$ word~$u$ is a length-$t$ binary word~$a(u)=a_1a_2\cdots a_t$ where $a_i=1$ if and only if $u$ has a border of length $t-i+1$. The \emph{auto-correlation polynomial} $f_{a(u)}(z)$ of $a(u)$ is defined as 
\[f_{a(u)}(z) = \sum_{i=0}^{t-1}a_{t-i}z^i.\]

For example, the word~$u={\tt entente}$ has auto-correlation $a(u) = 1001001$ and auto-correlation polynomial $f_{a(u)}(z) = z^6+z^3+1$.

%%%%%%%%%%%%%%%%%%%%%%%%%%%%%%%%%%%%%%%%%%%%%%%%%%%%%%%%%%%%%%%%%%%%%%%%%%%%%%%%%%%%%%%%%%%%%%%%%%%
%%%%%%%%%%%%%%%%%%%%%%%%%%%%%%%%%%%%%%%%%%%%%%%%%%%%%%%%%%%%%%%%%%%%%%%%%%%%%%%%%%%%%%%%%%%%%%%%%%%
\begin{comment}
\begin{theorem}[Section 7~\cite{Guibas&Odlyzko:1981}]
Let $k\geq 2$ and $t\geq 4$ be integers. Then
\[A_k(n,u) = C_u\theta_u^n + O(1.7^n)\]
where
\[\theta_u = k-\frac{1}{f_{a(u)}(k)} - \frac{f_{a(u)}'(k)}{(f_{a(u)}(k))^3}+O\bigg(\frac{t^2}{k^{3t}}\bigg)\text{ and }C_u = \frac{1}{1-(k-\theta_u)^2f_{a(u)}'(\theta_u)}.\]
\end{theorem}

\begin{theorem}[Section 7~\cite{Guibas&Odlyzko:1981}]
Let $k\geq 2$ be an integer. If $f_{a(u)}(2) > f_{a(v)}(2)$ for words $u$, $v$, then $\theta_{u} \geq \theta_v$ and $A_k(n,u) \geq A_k(n,v)$ for all $n\geq 1$.
\end{theorem}
\begin{corollary}
Let $k\geq 2$ and $t\geq 1$ be integers. Let $u=0^t$. Let $v$ be a length-$t$ word. Then $\theta_{u} \geq \theta_v$ and $A_k(n,u)\geq A_k(n,v)$ for all $n\geq 1$.
\end{corollary}

\begin{theorem}[Section 4~\cite{Rubinchik:2016}]
Let $k\geq 2$ and $t\geq 4$ be integers. Let $u=0^t$. Then
\[\theta_{u} = k-\frac{k-1}{k^t}-\frac{t(k-1)^2}{k^{2t+1}}+O\bigg(\frac{kt+t^2}{k^{3t}}\bigg)\text{ and }C_{u} = 1+O\bigg(\frac{t}{k^t}\bigg).\]
\end{theorem}
\end{comment}
%%%%%%%%%%%%%%%%%%%%%%%%%%%%%%%%%%%%%%%%%%%%%%%%%%%%%%%%%%%%%%%%%%%%%%%%%%%%%%%%%%%%%%%%%%%%%%%%%%%
%%%%%%%%%%%%%%%%%%%%%%%%%%%%%%%%%%%%%%%%%%%%%%%%%%%%%%%%%%%%%%%%%%%%%%%%%%%%%%%%%%%%%%%%%%%%%%%%%%%
%%%%%%%%%%%%%%%%%%%%%%%%%%%%%%%%%%%%%%%%%%%%%%%%%%%%%%%%%%%%%%%%%%%%%%%%%%%%%%%%%%%%%%%%%%%%%%%%%%%

We now prove two technical lemmas that will be used in the proofs of Theorem~\ref{theorem:mainC}~\ref{theorem:mainCUpper} and Theorem~\ref{theorem:mainP}~\ref{theorem:mainPUpper}.
%\begin{comment}
\begin{lemma}
\label{lemma:binomial}
Let $k,t\geq  2$ be integers, and let $\gamma$ be a real number such that $0< \gamma \leq \frac{6}{t}$. Then
\[k^t - \gamma tk^{t-1} \leq (k-\gamma)^t \leq k^t - \gamma t k^{t-1} + \frac{1}{2}\gamma^2 t(t-1) k^{t-2}.\]
\end{lemma}
\begin{proof}
 The case when $k=2$ was proved in a paper by Forsyth et al.~\cite[Lemma 9]{Forsyth2016}. We generalize their proof to $k\geq 3$.
 
 When $t=2$, we have $k^2 - 2k\gamma \leq  (k-\gamma)^2 \leq k^2 - 2k\gamma+\gamma^2 $. So suppose $t\geq 3$.  By the binomial theorem, we have 
 \begin{align}
     (k-\gamma)^t &= \sum_{i=0}^t k^{t-i}(-\gamma)^i{t \choose i} = k^t - \gamma t k^{t-1} + \sum_{i=2}^t k^{t-i}(-\gamma)^i{t \choose i} \nonumber \\
                  &\geq  k^t - \gamma t k^{t-1} + \sum_{j=1}^{\lfloor(t-1)/2\rfloor} \bigg(k^{t-2j}\gamma^{2j}{t \choose 2j} - k^{t-2j-1}\gamma^{2j+1}{t\choose 2j+1}\bigg).  \nonumber
 \end{align}
 So to show that $k^t - \gamma tk^{t-1} \leq (k-\gamma)^t$, it is sufficient to show that 
 \begin{equation}
     k^{t-2j}\gamma^{2j}{t \choose 2j} \geq k^{t-2j-1}\gamma^{2j+1}{t\choose 2j+1}
     \label{foo3}
 \end{equation}
 for $1\leq j \leq \lfloor (t-1)/2\rfloor \leq (t-1)/2$.
 
 By assumption we have that $\gamma \leq \frac{6}{t}$, so $\gamma \leq \frac{6}{t-2}$ and thus $\gamma t - 2\gamma \leq 6$. Adding $2\gamma -2$ to both sides we get $\gamma t -2 \leq 4+2\gamma$, and so $\frac{\gamma t - 2 }{\gamma +2}\leq 2$. If $i\geq 2 \geq \frac{\gamma t - 2 }{\gamma +2}$, then $(\gamma +2)i \geq \gamma t - 2$. This implies that $2(i+1) \geq \gamma(t-i)$, and \[\frac{k}{\gamma} \geq \frac{2}{\gamma} \geq \frac{t-i}{i+1} = \frac{{t \choose i+1}}{{t \choose i}}.\]
 Therefore letting $i=2j$, we have that $k{t\choose 2j} \geq \gamma {t\choose 2j+1}$. Multiplying both sides by $k^{t-2j-1}\gamma^{2j}$ we get $k^{t-2j}\gamma^{2j}{t\choose 2j} \geq k^{t-2j-1}\gamma^{2j+1}{t\choose 2j+1}$, which proves
 \eqref{foo3}.

 Now we prove that $(k-\gamma)^t \leq k^t - \gamma t k^{t-1} + \frac{1}{2}\gamma^2 t(t-1) k^{t-2}$. Going back to the binomial expansion of $(k-\gamma)^t$, we have 
 \begin{align}
     (k-\gamma)^t &= k^t - \gamma t k^{t-1} +\frac{1}{2}\gamma^2 t(t-1) k^{t-2}+\sum_{i=3}^t  k^{t-i}(-\gamma)^i{t \choose i}\nonumber \\
     &\leq k^t - \gamma t k^{t-1} + \frac{1}{2}\gamma^2 t(t-1) k^{t-2} \nonumber \\
     &- \sum_{j=1}^{\lfloor(t-2)/2\rfloor} \bigg(k^{t-2j-1}\gamma^{2j+1}{t \choose 2j+1} - k^{t-2j-2}\gamma^{2j+2}{t\choose 2j+2}\bigg).\nonumber 
 \end{align}
So to show that $(k-\gamma)^t \leq k^t - \gamma t k^{t-1} + \frac{1}{2}\gamma^2 t(t-1) k^{t-2}$, it is sufficient to show that \[k^{t-2j-1}\gamma^{2j+1}{t \choose 2j+1} \geq k^{t-2j-2}\gamma^{2j+2}{t\choose 2j+1}\] for $1\leq j \leq \lfloor (t-2)/2\rfloor$. But we have already proved that $k{t \choose i} \geq \gamma {t\choose i+1}$. Letting $i=2j$, we have that $k{t \choose 2j+1} \geq \gamma {t\choose 2j+2}$. Multiplying both sides by $k^{t-2j-2}\gamma^{2j+1}$ we get $k^{t-2j-1}\gamma^{2j+1}{t \choose 2j+1} \geq k^{t-2j-2}\gamma^{2j+2}{t\choose 2j+2}$. 

\end{proof}

\begin{lemma}\label{lemma:logLimit}
Let $i\geq 1$ and $k\geq 2$ be integers. Then for any constant $\gamma > 0$, we have \[\lim_{n\to \infty} \frac{\log_k^{\circ i}( n^\gamma)}{\log_k^{\circ i}(n)}=\begin{cases} 
      \gamma , & \text{if $i=1$;} \\
      1, & \text{if $i>1$.}
   \end{cases}\]
\end{lemma}
\begin{proof}
When $i=1$ we have $\lim\limits_{n\to \infty} \frac{\log_k^{}( n^\gamma)}{\log_k^{}( n)} =\gamma \lim\limits_{n\to \infty} \frac{\log_k^{}( n)}{\log_k^{}( n)}=\gamma$.

The proof is by induction on $i$. Since we will use L'H\^{o}pital's rule to evaluate the limit, we first compute the derivative of $\log_k^{\circ i}( n^\lambda)$ with respect to $n$ for any constant $\lambda >0$. We have \[\frac{d}{dn}\log_k^{\circ i}( n^\lambda) = \frac{\lambda}{n\prod\limits_{j=1}^{i-1}\log_k^{\circ j}(n^\lambda )}.\]
 In the base case, when $i=2$, we have
 \begin{align*}
     \lim_{n\to \infty} \frac{\log_k^{\circ 2}( n^\gamma)}{\log_k^{\circ 2}(n)} &= \lim_{n\to \infty} \frac{\frac{\gamma}{n \log_k(n^\gamma )}}{\frac{1}{n \log_k(n )}} = 1.
 \end{align*}
 Suppose $i>2$. Then we have
  \begin{align*}
     \lim_{n\to \infty} \frac{\log_k^{\circ i}( n^\gamma)}{\log_k^{\circ i}(n)} &= \lim_{n\to \infty} \frac{\frac{\gamma}{n\prod\limits_{j=1}^{i-1}\log_k^{\circ j}(n^\gamma )}}{\frac{1}{n\prod\limits_{j=1}^{i-1}\log_k^{\circ j}(n)}} =\lim_{n\to \infty} \frac{ {\prod\limits_{j=2}^{i-1}\log_k^{\circ j}(n)}}{\prod\limits_{j=2}^{i-1}\log_k^{\circ j}(n^\gamma )} = 1.
 \end{align*}
\end{proof}

%\end{comment}

\section{Closed words}

\subsection{Lower bound}

We first state a useful lemma from a paper of Nicholson and Rampersad~\cite{Nicholson2018}.

\begin{lemma}[Nicholson and Rampersad~\cite{Nicholson2018}]\label{lemma:nich&ramp}
Let $k\geq 2$ be an integer. For every $n$, there is a unique integer $t$ such that \[\frac{\ln k}{k-1}k^t\leq n-t<  \frac{\ln k}{k-1}k^{t+1}.\] Let $u$ be a length-$t$ word. There exist constants $N_0$ and $d$ such that for $n-t> N_0$ we have \[B_k(n,u) \geq d\frac{k^n}{n^2}.\]
\end{lemma}

We now use the previous lemma to prove Theorem~\ref{theorem:mainC}~\ref{theorem:mainCLower}.
\begin{proof}[Proof of Theorem~\ref{theorem:mainC}~\ref{theorem:mainCLower}]
The number $C_k(n,t)$ of length-$n$ words closed by a length-$t$ word is clearly equal to the sum, over all length-$t$ words $u$, of the number $B_k(n,u)$ of length-$n$ words closed by $u$. Thus we have that \[C_k(n,t) = \sum_{|u| = t}B_k(n,u).\] Let $t$ be such that $t= \lfloor \log_k(n-t) + \log_k(k-1) - \log_k(\ln k )\rfloor$. By Lemma~\ref{lemma:nich&ramp} there exist constants $N_0$ and $d$ such that for $n-t > N_0$ we have $B_k(n,u) \geq d k^n/n^2$. Clearly $t \leq \log_k(n)+1$ for all $n\geq 1$. Since $t$ is asymptotically much smaller than $n$, there exists a constant $N> N_0$ such that $n-t > N_0$ for all $n>N$. Thus for $n>N$ we have
\begin{align*}
    C_k(n) &\geq C_k(n,t) = \sum_{|u| = t}B_k(n,u) \geq \sum_{|u|=t}d\frac{k^n}{n^2} = k^{t}\bigg(d  \frac{k^n}{n^2}\bigg)  \\
           &=d k^{\lfloor \log_k(n-t) + \log_k(k-1) - \log_k(\ln k )\rfloor} \frac{k^n}{n^2} \geq d_0k^{\log_k(n-t) + \log_k(k-1) - \log_k(\ln k )}  \frac{k^n}{n^2} \\
           &\geq  d_1(n-t)\frac{k^n}{n^2} \geq d_1(n-\log_k(n)-1)\frac{k^n}{n^2}\geq c\frac{k^n}{n}
\end{align*}
for some constant $c>0$.
\end{proof}

\subsection{Upper bound}

Before we proceed with upper bounding $C_k(n)$, we briefly outline the direction of the proof. First, we begin by bounding $C_k(n,t)$ for $t< n/2$ and $t\geq n/2$. We show that for $t< n/2$, the number of length-$n$ words closed by a particular length-$t$ word~$u$ is bounded by the number of words of length $n-2t$ that do not have $0^t$ as a factor. For $t\geq n/2$ we prove that $C_k(n,t)$ is negligibly small. Next, we prove upper bounds on the number of words that do not have $0^t$ as a factor, allowing us to finally bound $C_k(n)$.

\begin{lemma}\label{lemma:BboundA}
Let $n$, $t$, and $k$ be integers such that $n\geq 2t\geq 2$ and $k\geq 2$. Let $u$ be a length-$t$ word. Then \[B_k(n,u) \leq A_k(n-2t,0^t).\]
\end{lemma}
\begin{proof}
Recall that $B_k(n,u)$ is the number of length-$n$ words that are closed by the word~$u$. Also recall that $A_k(n,u)$ is the number of length-$n$ words that do not contain the word $u$ as a factor.

 Let $w$ be a length-$n$ word closed by $u$ where $|w| = n \geq 2t = 2|u|$. Then we can write $w=uvu$ where $v$ does not contain $u$ as a factor. This immediately implies that $B_k(n,u) \leq A_k(n-2t,u)$. But from a result of Guibas and Odlyzko~\cite[Section 7]{Guibas&Odlyzko:1981}, we have that if $f_{a(u)}(2) > f_{a(v)}(2)$ for words $u$, $v$, then $A_k(m,u) \geq A_k(m,v)$ for all $m\geq 1$. The auto-correlation polynomial only has $0$ or $1$ as coefficients, depending on the $1$'s and $0$'s in the auto-correlation. Thus, the auto-correlation $p$ that maximizes $f_p(2)$ is clearly $p=1^t$. The words that achieve this auto-correlation are words of the form $a^t$ where $a\in \Sigma_k$. Therefore we have 
 \[B_k(n,u) \leq A_k(n-2t,u) \leq A_k(n-2t,0^t).\]
\end{proof}

\begin{lemma}\label{lemma:Cntbound}
Let $n$, $t$, and $k$ be integers such that $n\geq 2t\geq 2$ and $k\geq 2$. Then
\[C_k(n,t) \leq k^t A_k(n-2t,0^t).\]
\end{lemma}
\begin{proof}
The number $C_k(n,t)$ of length-$n$ words closed by a length-$t$ word is equal to the sum, over all length-$t$ words $u$, of the number $B_k(n,u)$ of length-$n$ words closed by $u$. Thus we have that \[C_k(n,t) = \sum_{|u| = t}B_k(n,u).\]
By Lemma~\ref{lemma:BboundA} we have that $B_k(n,v) \leq A_k(n-2t,0^t)$ for all length-$t$ words $v$. Therefore \[C_k(n,t)=\sum_{|u| = t}B_k(n,u)  \leq  \sum_{|u|=t} A_k(n-2t, 0^t)  \leq k^tA_k(n-2t,0^t).\]

\end{proof}
\begin{corollary}\label{corollary:Cbound}
Let $n\geq 1$ and $k\geq 2$ integers. Then 
\[C_k(n) \leq \sum_{t=1}^{\lfloor n/2\rfloor} k^t A_k(n-2t,0^t) + nk^{\lceil n/2\rceil}.\]
\end{corollary}
\begin{proof}
It follows from Lemma~\ref{lemma:Cntbound} that \[C_k(n)=\sum_{t=1}^{n-1}C_k(n,t) \leq \sum_{t=1}^{\lfloor n/2\rfloor} k^t A_k(n-2t,0^t) + \sum_{t=\lfloor n/2\rfloor +1}^{n-1} C_k(n,t).\]
Now we show that \[\sum_{t=\lfloor n/2\rfloor +1}^{n-1} C_k(n,t)\leq nk^{\lceil n/2\rceil}.\]
Let $w=w_0w_1\cdots w_{n-1}$ be a word of length $n$ that is closed by a word~$u$ of length $t>\lfloor n/2\rfloor $. Then $w = ux = yu$ for some words $x$, $y$. So $w_i = w_{i+(n-t)}$ for all $i$, $0\leq i < t$. This implies that $w= v^i v'$ where $v$ is the length-$(n-t)$ prefix of $w$, $i=\lfloor n/|v|\rfloor$, and $v'$ is the length-$(n-i|v|)$ prefix of $v$. Since $t > \lfloor n/2 \rfloor$, we have that $n-t < \lceil n/2\rceil$. We see that $w$ is fully determined by the word~$v$. So since $|v| < \lceil n/2\rceil$, we have $C_k(n,t) \leq k^{\lceil n/2\rceil}$. Thus
\[\sum_{t=\lfloor n/2\rfloor +1}^{n-1} C_k(n,t)\leq  \sum_{t=\lfloor n/2\rfloor +1}^{n-1} k^{\lceil n/2\rceil} \leq nk^{\lceil n/2\rceil}.\]
\end{proof}

%%%%%%%%%%%%%%%%%%%%%%%%%%%%%%%%%%%%%%%%%%%%%%%%%%%%%%%%%%%%%%%%%%%%%%%%%%%%%%%%%%%%%%%%%%%%%%%%%%
%%%%%%%%%%%%%%%%%%%%%%%%%%%%%%%%%%%%%%%%%%%%%%%%%%%%%%%%%%%%%%%%%%%%%%%%%%%%%%%%%%%%%%%%%%%%%%%%%%
%%%%%%%%%%%%%%%%%%%%%%%%%%%%%%%%%%%%%%%%%%%%%%%%%%%%%%%%%%%%%%%%%%%%%%%%%%%%%%%%%%%%%%%%%%%%%%%%%%
%%%%%%%%%%%%%%%%%%%%%%%%%%%%%%%%%%%%%%%%%%%%%%%%%%%%%%%%%%%%%%%%%%%%%%%%%%%%%%%%%%%%%%%%%%%%%%%%%%
%\begin{comment}
\begin{lemma}\label{lemma:A0t}
Let $n\geq 0$, $t\geq 1$, and $k\geq 2$ be integers. Then
\[ A_k(n,0^t) = \begin{cases} 
      k^n, & \text{if $n < t$;} \\
      (k-1)\sum\limits_{i=1}^t A_k(n-i,0^t), & \text{if $n\geq t$.}
   \end{cases}
\]
\end{lemma}
\begin{proof}
If $n<t$, then any length-$n$ word is shorter than $0^t$, and thus cannot contain $0^t$ as a factor. So $A_k(n,0^t)=k^n$.

Suppose $n\geq t$. Let $w$ be a length-$n$ word that does not contain $0^t$ as a factor. Let us look at the symbols that $w$ ends in. Since $w$ does not contain $0^t$, we have that $w$ ends in anywhere from $0$ to $t-1$ zeroes. So $w$ is of the form $w = w' b 0^i$ where $i$ is an integer with $0\leq i \leq t-1$, $b\in \Sigma_k-\{0\}$, and $w'$ is a length-$(n-i-1)$ word that does not contain $0^t$ as a factor. There are $k-1$ choices for $b$, and $A_k(n-i-1,0^t)$ choices for $w'$. So there are $(k-1)A_k(n-i-1,0^t)$ words of the form $w'b0^i$. Summing over all possible $i$ gives 
\[A_k(n,0^t)=(k-1)\sum_{i=1}^t A_k(n-i,0^t).\]
\end{proof}

\begin{corollary}
Let $n\geq 0$, $t\geq 1$, and $k\geq 2$ be integers. Then
\[ A_k(n,0^t) = \begin{cases} 
      k^n, & \text{if $n < t$;} \\
      k^n-1, & \text{if $n=t$;}\\
      kA_k(n-1,0^t) - (k-1)A_k(n-t-1,0^t), & \text{if $n> t$.}
   \end{cases}
\]
\end{corollary}
\begin{proof}
Compute $A_k(n,0^t)-A_k(n-1,0^t)$ with the recurrence from Lemma~\ref{lemma:A0t} and the result follows.
\end{proof}
\begin{comment}
\begin{proof}
Suppose $n=t$. Then \[A_k(n,0^t) = (k-1)\sum_{i=1}^t A_k(t-i,0^t) = (k-1)\sum_{i=1}^tk^{t-i} = k^n-1.\]
Suppose $n>t$. Then 
\begin{align*}
    A_k(n,0^t) - A_k(n-1,0^t) &= (k-1)\sum_{i=1}^t A_k(n-i,0^t) - (k-1)\sum_{i=1}^t A_k(n-1-i,0^t)\\ &= (k-1)A_k(n-1,0^t)-(k-1)A_k(n-t-1,0^t).
\end{align*}
Therefore \[A_k(n,0^t) =kA_k(n-1,0^t) - (k-1)A_k(n-t-1,0^t).\]
\end{proof}
\end{comment}

\begin{corollary}\label{corollary:recUpp}
Let $n\geq 0$, $t\geq 1$, and $k\geq 2$ be integers. Then
\[ A_k(n,0^t) = \begin{cases} 
      k^n, & \text{if $n < t$;} \\
      k^{n-t}(k^t -1)-(n-t)k^{n-t-1}(k-1), & \text{if $t \leq n \leq 2t$;}\\
      kA_k(n-1,0^t) - (k-1)A_k(n-t-1,0^t), & \text{if $n> 2t$.}
   \end{cases}
\]
\end{corollary}
\begin{proof}
We prove $t\leq n \leq 2t$ by induction on $n$. In the base case, when $n=t$, we have $k^t-1 = A_k(t,0^t) = k^{t-t}(k^t-1) - (t-t)k^{t-t-1}(k-1) = k^t-1$. 

Suppose $t < n \leq 2t$. Then \begin{align}
    A_k(n,0^t) &= kA_k(n-1,0^t) - (k-1)A_k(n-t-1,0^t) \nonumber \\
               &= k(k^{n-1-t}(k^t-1) - (n-1-t)k^{n-t-2}(k-1)) - (k-1)k^{n-t-1}\nonumber \\
               &=k^{n-t}(k^t-1) - (n-t)k^{n-t-1}(k-1).\nonumber
\end{align} 
\end{proof}

Since $(A_k(n,0^{t}))_n$ satisfies a linear recurrence, we know that the asymptotic behaviour of $A_k(n,0^t)$ is determined by the root of maximum modulus of the polynomial $x^{t+1}-kx^t+k-1=0$. We use this fact to find an upper bound for $A_k(n,0^t)$.

\begin{lemma}
\label{theorem:betaUpper}
Let $t\geq 1$ and $k\geq 2$ be integers. Let \[\beta_k(t) = k-(k-1)k^{-t-1}.\] Then $\beta_k(t) \geq k-(k-1)\beta_k(t)^{-t}$.
\end{lemma}
\begin{proof}
Since $\beta_k(t) \leq k$, we have that $\beta_k(t)^{-t} \geq k^{-t}\geq {k^{-t-1}}$. This implies that 
\[\beta_k(t) = k-(k-1)k^{-t-1} \geq k-(k-1)\beta_k(t)^{-t}.\]
\end{proof}

\begin{lemma}
\label{lemma:baseCaseUpper}
Let $k,t\geq 2$ be integers. Let $n$ be an integer such that $2t \leq n\leq 3t$. Then $A_k(n,0^t) \leq \beta_k(t)^n$.
\end{lemma}
\begin{proof}
The proof is by induction on $n$. By Corollary~\ref{corollary:recUpp} we have that \[A_k(n,0^t) = k^{n-t}(k^t -1)-(n-t)k^{n-t-1}(k-1)\] for $t \leq n \leq 2t$.

Suppose, for the base case, that $n=2t$. Let $\gamma(t) = (k-1)k^{-t-1}$. Then
\begin{align*}
    A_k(2t,0^t) &= k^{t}(k^t -1)-tk^{t-1}(k-1) = k^{2t} - k^{t-2}(k^2 + tk(k-1))\\
               &= k^{2t} - \gamma(t) k^{2t-1}\frac{(k^2 + tk(k-1))}{k-1} \\
               &\leq k^{2t} - \gamma(t) t k^{2t-1}.
\end{align*}
Clearly $\gamma(t) \leq 6/t$ for all $t\geq 2$, so $A_k(2t)\leq k^{2t} - \gamma(t) t k^{2t-1} \leq (k- \gamma(t))^{2t} = \beta_k(t)^{2t}$.

Suppose that $2t < n \leq 3t$. Furthermore let $n = 2t + i + 1$ where $i$ is an integer such that $0\leq i < t$. Notice that $A_k(n-t-1,0^t) = A_k(t+i,0^t) = k^{i}(k^t-1) - ik^{i-1}(k-1)$. Then
{\small
\begin{align*}
    A_k(2t+i+1,0^t) &= kA_k(2t+i,0^t) - (k-1)A_k(t+i,0^t)\\
                    &\leq k(k-\gamma(t))^{2t+i} - (k-1)(k^{i}(k^t-1) - ik^{i-1}(k-1)) \\
                    &= (k-\gamma(t))^{2t+i+1} + \gamma(t)(k-\gamma(t))^{2t+i}- (k-1)(k^{i}(k^t-1) - ik^{i-1}(k-1))\\
                    &= \beta_k(t)^{2t+i+1} + \gamma(t)\beta_k(t)^{2t+i}- (k-1)(k^{i}(k^t-1) - ik^{i-1}(k-1)).
\end{align*}}%
To prove the desired bound, namely that $A_k(2t+i+1,0^t) \leq \beta_k(t)^{2t+i+1}$, it is sufficient to show that $\beta_k(t)^{2t+i} \leq \gamma(t)^{-1}(k-1)(k^{i}(k^t-1) - ik^{i-1}(k-1))$. We begin by upper bounding $\beta_k(t)^{2t+i}$ with Lemma~\ref{lemma:binomial}. We have
\begin{align}
    \beta_k(t)^{2t+i} &\leq k^{2t+i} - \gamma(t)(2t+i)k^{2t+i-1}+\frac{1}{2}\gamma(t)^2(2t+i)(2t+i-1)k^{2t+i-2} \nonumber \\
                    &\leq k^{2t+i} - 2(k-1)tk^{t+i-2}+\frac{9}{2}(k-1)^2t^2k^{i-4} \nonumber\\
                    &\leq k^{2t+i+1} -(k-1)k^{2t+i} - 2(k-1)tk^{t+i-2}+\frac{9}{2}(k-1)^2t^2k^{i-4} \nonumber\\
                    &= k^{2t+i+1} -k^{t+i}\Big((k-1)k^t + 2(k-1)tk^{-2} -\frac{9}{2}(k-1)^2t^2k^{-t-4} \Big) . \label{lastLine}
\end{align}
It is easy to verify that $(k-1)k^t \geq k+t(k-1)$ and $2(k-1)tk^{-2} -\frac{9}{2}(k-1)^2t^2k^{-t-4} \geq 0$ for all $t\geq 2$. Thus, continuing from (\ref{lastLine}), we have
\begin{align*}
    \beta_k(t)^{2t+i} &\leq k^{2t+i+1} -k^{t+i}(k + t(k-1) )  \leq  k^{2t+i+1} -k^{t+i}(k + i(k-1) ) \\
                    &= \frac{k^{t+1}}{k-1}(k-1)(k^{t+i} - k^{i} - ik^{i-1}(k-1))\\
                    &= \gamma(t)^{-1} (k-1)(k^i(k^t-1) - ik^{i-1}(k-1)).
\end{align*}
\end{proof}

\begin{lemma}
\label{lemma:upperBoundA0}
Let $n$, $t$, and $k$ be integers such that $n\geq 2t \geq 4$ and $k\geq 2$. Then $A_k(n,0^t) \leq \beta_k(t)^{n}$.
\end{lemma}

\begin{proof}
The proof is by induction on $n$. The base case, when $2t\leq n \leq 3t$, is taken care of in Lemma~\ref{lemma:baseCaseUpper}.

Suppose $n>3t$. Then 
\[
    A_k(n,0^t) = (k-1)\sum_{i=1}^t A_k(n-i,0^t)\leq  (k-1)\sum_{i=1}^t \beta_k(t)^{n-i}= (k-1)\frac{\beta_k(t)^n - \beta_k(t)^{n-t}}{\beta_k(t) - 1}.\]
By Theorem~\ref{theorem:betaUpper}, we have that $\beta_k(t) -1 \geq (k-1)-(k-1)\beta_k(t)^{-t}$. Therefore 
\[A_k(n,0^t) \leq (k-1)\frac{\beta_k(t)^n - \beta_k(t)^{n-t}}{\beta_k(t) - 1}= \beta_k(t)^n \frac{(k-1) - (k-1)\beta_k(t)^{-t}}{\beta_k(t)-1} \leq \beta_k(t)^n.\]
\end{proof}
%\end{comment}
%%%%%%%%%%%%%%%%%%%%%%%%%%%%%%%%%%%%%%%%%%%%%%%%%%%%%%%%%%%%%%%%%%%%%%%%%%%%%%%%%%%%%%%%%%%%%%%%%%
%%%%%%%%%%%%%%%%%%%%%%%%%%%%%%%%%%%%%%%%%%%%%%%%%%%%%%%%%%%%%%%%%%%%%%%%%%%%%%%%%%%%%%%%%%%%%%%%%%
%%%%%%%%%%%%%%%%%%%%%%%%%%%%%%%%%%%%%%%%%%%%%%%%%%%%%%%%%%%%%%%%%%%%%%%%%%%%%%%%%%%%%%%%%%%%%%%%%%
%%%%%%%%%%%%%%%%%%%%%%%%%%%%%%%%%%%%%%%%%%%%%%%%%%%%%%%%%%%%%%%%%%%%%%%%%%%%%%%%%%%%%%%%%%%%%%%%%%

\begin{proof}[Proof of Theorem~\ref{theorem:mainC}~\ref{theorem:mainCUpper}]
First notice that $A_k(n,0) = (k-1)^n$, since $A_k(n,0)$ is just the number of length-$n$ words that do not contain $0$.

Let $N'$ be a positive integer such that the following inequalities hold for all $n>N'$.

\begin{align}
    C_k(n) &\leq \sum_{t=2}^{\lfloor n/2\rfloor} k^t A_k(n-2t,0^t) + kA_k(n-2,0) + nk^{\lceil n/2\rceil} \nonumber \\
           &\leq \sum_{t=2}^{\lfloor n/2\rfloor} k^t \beta_k(t)^{n-2t} + k(k-1)^{n-2} + nk^{\lceil n/2\rceil}\nonumber \\
          &\leq  \sum_{t=2}^{\lfloor n/2\rfloor} k^t \bigg(k-\frac{k-1}{k^{t+1}}\bigg)^{n-2t} + d_2\frac{k^n}{n}= k^n\sum_{t=2}^{\lfloor n/2\rfloor} \frac{1}{k^{t}} \bigg(1-\frac{k-1}{k^{t+2}}\bigg)^{n-2t} + d_2\frac{k^n}{n}\nonumber \\
          &\leq k^n\Bigg(\sum_{t=2}^{\lfloor \log_k n\rfloor} \frac{1}{k^{t}} \bigg(1-\frac{k-1}{k^{t+2}}\bigg)^{n-2t} +\sum_{t=\lfloor \log_k n\rfloor+1}^{\lfloor n/2\rfloor} \frac{1}{k^{t}}\Bigg) + d_2\frac{k^n}{n}\nonumber \\
          &\leq k^n\Bigg(\sum_{t=2}^{\lfloor \log_k n\rfloor } \frac{1}{k^{t}} \bigg(1-\frac{k-1}{k^{t+2}}\bigg)^{n-2\lfloor \log_k n\rfloor}+   \frac{d_3}{n}\Bigg) + d_2\frac{k^n}{n}\nonumber\\
          &\leq k^n\sum_{t=2}^{\lfloor \log_k n\rfloor } \frac{1}{k^{t}} \bigg(1-\frac{k-1}{k^{t+2}}\bigg)^{n/2}+   d_4\frac{k^n}{n}. \label{firstSum}
\end{align}
Now we bound the sum in~(\ref{firstSum}).
Let $h(x) = (1-(k-1)k^{-2}x)^{n/2}$. Notice that $h(x)$ is monotonically decreasing on the interval $x\in (0,1)$. So for $k^{-t-1} \leq x \leq k^{-t}$ we have that $h(x) \geq h(k^{-t})$. Thus
\[\frac{1}{k^{t}} \bigg(1-\frac{k-1}{k^{t+2}}\bigg)^{n/2}\leq \frac{k-1}{k^{t}} \bigg(1-\frac{k-1}{k^{t+2}}\bigg)^{n/2} \leq k\bigg(\bigg(\frac{1}{k^{t}}-\frac{1}{k^{t+1}}\bigg)h(k^{-t})\bigg) \leq k\int_{k^{-t-1}}^{k^{-t}} h(x) \, dx.\] Going back to (\ref{firstSum}) we have
\[C_k(n) \leq k^n\sum_{t=2}^{\lfloor \log_k n\rfloor } k\int_{k^{-t-1}}^{k^{-t}} h(x) \, dx+   d_4\frac{k^n}{n} \leq k^{n+1}\int_0^1 h(x) \, dx+   d_4\frac{k^n}{n}.\]
Evaluating and bounding the definite integral, we have
\begin{align*}
    \int_0^1 h(x) \, dx &= -\frac{k^2}{k-1}\bigg[\frac{(1-(k-1)k^{-2}x)^{n/2+1}}{n/2+1}\bigg]_{x=0}^{x=1}\\
    &= -\frac{k^2}{k-1}\bigg(\frac{(1-(k-1)k^{-2})^{n/2+1} -1}{n/2+1}\bigg) \\
    &\leq d_5\bigg(\frac{1-(1-(k-1)k^{-2})^{n/2+1}}{n/2+1}\bigg) \leq d_5\frac{1}{n/2+1} \leq \frac{d_6}{n}. \\
\end{align*}
Putting everything together, we have that
\[C_k(n) \leq k^{n+1}\int_0^1 h(x) \, dx+   d_4\frac{k^n}{n} \leq d_6\frac{k^{n+1}}{n} + d_4\frac{k^n}{n} \leq c'\frac{k^n}{n}\]
for some constant $c'>0$.
\end{proof}

\section{Privileged words}\label{section:privileged}
\subsection{Lower bound}

In this section we provide a family of lower bounds for the number of length-$n$ privileged words. We use induction to prove these bounds. The basic idea is that we start with the lower bound by Nicholson and Rampersad, and then use it to bootstrap ourselves to better and better lower bounds.

\begin{proof}[Proof of Theorem~\ref{theorem:mainP}~\ref{theorem:mainPLower}]
The proof is by induction on $j$. Let $t$ be such that $t= \lfloor \log_k(n-t) + \log_k(k-1) - \log_k(\ln k )\rfloor$. We clearly have $0\leq t \leq \log_k(n) + 1$ for all $n\geq 1$. Let $u$ be a length-$t$ privileged word. By Lemma~\ref{lemma:nich&ramp} we have that there exist constants $N_0$ and $c_0$ such that $P_k(n)\geq B_k(n,u) \geq c_0\frac{k^n}{n^2}$ for all $n > N_0$. So the base case, when $j=0$, is taken care of.

Suppose $j>0$. By induction we have that there exist constants $N_{j-1}$ and $c_{j-1}$ such that \[P_k(n)\geq c_{j-1}\frac{k^n}{n\log_k^{\circ j-1}(n)\prod_{i=1}^{j-1}\log_k^{\circ i}(n)}\] for all $n>N_{j-1}$. By Lemma~\ref{lemma:nich&ramp} we have
\begin{align*}
    P_k(n) \geq P_k(n,t) \geq \sum_{\substack{|u|=t\\u\text{ privileged}}}B_k(n,u) \geq \sum_{\substack{|u|=t\\u\text{ privileged}}}d\frac{k^n}{n^2} =dP_k(t) \frac{k^n}{n^2}.
\end{align*}
for $n>N_0$. Since $t \leq \log_k(n) +1$, we have that $\frac{1}{\log_k^{\circ i}(t)}\geq \frac{1}{\log_k^{\circ i}(\log_k(n)+1)}$ for all $i\geq 0$. Thus continuing from above we have
\begin{align*}
    P_k(n) &\geq dc_{j-1}\frac{k^t}{t\log_k^{\circ j-1}(t)\prod_{i=1}^{j-1}\log_k^{\circ i}(t)} \frac{k^n}{n^2}\geq d_7\frac{k^{\log_k(n-t) + \log_k(k-1) - \log_k(\ln k )}}{t\log_k^{\circ j-1}(t)\prod_{i=1}^{j-1}\log_k^{\circ i}(t)}\frac{k^n}{n^2} \\
    &\geq d_8\frac{1}{t\log_k^{\circ j-1}(t)\prod_{i=1}^{j-1}\log_k^{\circ i}(t)}\frac{k^n}{n}\\
    &\geq d_9\frac{1}{(\log_k(n)+1)\log_k^{\circ j-1}(\log_k(n)+1)\prod_{i=1}^{j-1}\log_k^{\circ i}(\log_k(n)+1)}\frac{k^n}{n} \\
    &\geq c_j\frac{k^n}{n\log_k^{\circ j}(n)\prod_{i=1}^{j}\log_k^{\circ i}(n)}
\end{align*}
for all $n>N_j$ where $N_j > \max(N_0, N_{j-1})$.
\end{proof}

\subsection{Upper bound}
In Theorem~\ref{theorem:mainC}~\ref{theorem:mainCUpper} we proved that $C_k(n) \in O(\frac{k^n}{n})$. Since every privileged word is also a closed word, this is also shows that $P_k(n) \in O(\frac{k^n}{n})$. This bound improves on the existing bound on privileged words but it does not show that $P_k(n)$ and $C_k(n)$ behave differently asymptotically. We show that $P_k(n)$ is much smaller than $C_k(n)$ asymptotically by proving upper bounds on $P_k(n)$ that show $P_k(n) \in o(\frac{k^n}{n})$.
\begin{lemma}\label{lemma:Pntbound}
Let $n$, $t$, and $k$ be integers such that $n\geq 2t\geq 2$ and $k\geq 2$. Then
\[P_k(n,t) \leq P_k(t)A_k(n-2t,0^t).\]
\end{lemma}
\begin{proof}
The number of length-$n$ privileged words closed by a length-$t$ privileged word is equal to the sum, over all length-$t$ privileged words $u$, of the number $B_k(n,u)$ of length-$n$ words closed by $u$. Thus we have that \[P_k(n,t) = \sum_{\substack{|u|=t\\u\text{ privileged}}}B_k(n,u).\]
By Lemma~\ref{lemma:BboundA} we have that $B_k(n,v) \leq A_k(n-2t,0^t)$ for all length-$t$ words $v$. Therefore
\[ P_k(n,t) = \sum_{\substack{|u|=t\\u\text{ privileged}}}B_k(n,u)  \leq \sum_{\substack{|u|=t\\u\text{ privileged}}}A_k(n-2t,0^t) \leq P_k(t) A_k(n-2t,0^t).\]
\end{proof}

\begin{proof}[Proof of Theorem~\ref{theorem:mainP}~\ref{theorem:mainPUpper}]
For $n\geq 2t$ we can use Lemma~\ref{lemma:Pntbound} to bound $P_k(n,t)$. But for $n < 2t$, we can use Corollary~\ref{corollary:Cbound} and the fact that $P_k(n,t) \leq C_k(n,t)$. We get
\[P_k(n) = \sum_{t=1}^{n-1} P_k(n,t) \leq \sum_{t=1}^{\lfloor n/2\rfloor} P_k(t) A_k(n-2t,0^t) + nk^{\lceil n/2\rceil}.\]
The proof is by induction on $j$. The base case, when $j=0$, is taken care of by Theorem~\ref{theorem:mainC}~\ref{theorem:mainCUpper}. 

Suppose $j>0$. Then there exist constants $N_{j-1}'$ and $c_{j-1}'$ such that \[P_k(n) \leq c_{j-1}'\frac{k^n}{n\prod_{i=1}^{j-1}\log_k^{\circ i}(n)}\] for all $n>N_{j-1}'$.
We now bound $P_k(n)$. First, we let $N_j' > N_{j-1}'$ be a constant such that the following inequalities hold for all $n>N_j'$. We have

{\small
\begin{align}
    P_k(n) &\leq \sum_{t=1}^{\lfloor n/2\rfloor} P_k(t) A_k(n-2t,0^t)  + nk^{\lceil n/2\rceil} \nonumber \\
           &\leq \sum_{t=N_{j}'+1}^{\lfloor n/2\rfloor} c_{j-1}'\frac{k^t}{t\prod_{i=1}^{j-1}\log_k^{\circ i}(t)} \beta_k(t)^{n-2t} + \sum_{t=1}^{N_{j}'} P_k(t) A_k(n-2t,0^t) + nk^{\lceil n/2\rceil} \nonumber \\
           &\leq \sum_{t=N_{j}'+1}^{\lfloor n/2\rfloor} c_{j-1}'\frac{k^t}{t\prod_{i=1}^{j-1}\log_k^{\circ i}(t)}\bigg(k-\frac{k-1}{k^{t+1}}\bigg)^{n-2t} + d_{10}\sum_{t=2}^{N_{j}'} \bigg(k-\frac{k-1}{k^{t+1}}\bigg)^{n-2t} + d_{11}\frac{k^n}{n^2} \nonumber \\
           &\leq c_{j-1}'k^n\sum_{t=N_{j}'+1}^{\lfloor n/2\rfloor} \frac{1}{k^tt\prod_{i=1}^{j-1}\log_k^{\circ i}(t)} \bigg(1-\frac{k-1}{k^{t+2}}\bigg)^{n-2t} +  d_{12}\frac{k^n}{n^2} \nonumber\\
           &\leq c_{j-1}'k^n\bigg(d_{13}\sum_{t=N_{j}'+1}^{\lfloor \log_k(n)\rfloor} \frac{1}{k^tt\prod_{i=1}^{j-1}\log_k^{\circ i}(t)} \bigg(1-\frac{k-1}{k^{t+2}}\bigg)^{n/2}\nonumber \\
           &\hspace*{2cm}+\sum_{t=\lfloor \log_k(n)\rfloor+1}^{\lfloor n/2\rfloor} \frac{1}{k^tt\prod_{i=1}^{j-1}\log_k^{\circ i}(t)} \bigg) +  d_{12}\frac{k^n}{n^2} \nonumber \\
           &\leq c_{j-1}'k^n\bigg(d_{13}\sum_{t=N_{j}'+1}^{\lfloor \log_k(n)\rfloor} \frac{1}{k^tt\prod_{i=1}^{j-1}\log_k^{\circ i}(t)} \exp{\bigg(\frac{n}{2}\ln\bigg(1-\frac{k-1}{k^{t+2}}\bigg)\bigg)}\nonumber \\
           &\hspace*{2cm}+\sum_{t=\lfloor \log_k(n)\rfloor+1}^{\infty} \frac{1}{k^tt\prod_{i=1}^{j-1}\log_k^{\circ i}(t)} \bigg) +  d_{12}\frac{k^n}{n^2}. \label{Pline}
\end{align}}%
The sum on line~(\ref{Pline}) is clearly convergent. We have 
{
\begin{align*}
\sum_{t=\lfloor \log_k(n)\rfloor+1}^{\infty}\frac{1}{k^tt\prod_{i=1}^{j-1}\log_k^{\circ i}(t)} &\leq \frac{1}{(\lfloor \log_k(n)\rfloor+1)\prod_{i=1}^{j-1}\log_k^{\circ i}(\lfloor \log_k(n)\rfloor+1)}\sum_{t=\lfloor \log_k(n)\rfloor+1}^{\infty}\frac{1}{k^t} \\
 &\leq d_{14}\frac{1}{\log_k(n)\prod_{i=1}^{j-1}\log_k^{\circ i}(\log_k(n))}\frac{1}{n} \leq d_{14}\frac{1}{n\prod_{i=1}^{j}\log_k^{\circ i}(n)}.
\end{align*}}%
 Now we upper bound the sum \[D_n= \sum_{t=N_{j}'+1}^{\lfloor \log_k(n)\rfloor} \frac{1}{k^tt\prod_{i=1}^{j-1}\log_k^{\circ i}(t)} \exp{\bigg(\frac{n}{2}\ln\bigg(1-\frac{k-1}{k^{t+2}}\bigg)\bigg)}.\] It is well-known that $\ln(1-x) \leq -x$ for $|x| < 1$. Thus, letting $\alpha= \frac{k-1}{2k^2}$, we have \[\exp{\bigg(\frac{n}{2}\ln\bigg(1-\frac{k-1}{k^{t+2}}\bigg)\bigg)} \leq \exp{\Big(-\alpha\frac{n}{k^{t}}\Big)}.\] 
We reverse the order of the series, by letting $t$ be such that $t=\lfloor \log_k(n)\rfloor -t+N_{j}'+1$. We also shift the index of the series down by $N_{j}'+1$. We have
{\footnotesize
\begin{align}
       D_n   &= \sum_{t=0}^{\lfloor \log_k(n)\rfloor-N_{j}' - 1} \frac{1}{k^{\lfloor \log_k(n)\rfloor -t}(\lfloor \log_k(n)\rfloor -t)\prod_{i=1}^{j-1}\log_k^{\circ i}(\lfloor \log_k(n)\rfloor -t)}\exp{\Big(-\alpha\frac{n}{k^{\lfloor \log_k(n)\rfloor -t}}\Big)} \nonumber \\
           &\leq d_{15}\sum_{t=0}^{\lfloor \log_k(n)\rfloor-N_{j}' - 1} \frac{k^{t}}{n(\log_k(n)-t)\prod_{i=1}^{j-1}\log_k^{\circ i}(\log_k(n) -t)} \exp{(-\alpha k^t)} \nonumber\\
           &\leq d_{15}\frac{1}{n\prod_{i=1}^{j}\log_k^{\circ i}(n)}\sum_{t=0}^{\lfloor \log_k(n)\rfloor-N_{j}' - 1} \frac{k^{t}}{\prod\limits_{i=0}^{j-1}\frac{\log_k^{\circ i}(\log_k(n) -t)}{\log_k^{\circ i+1}(n)}} \exp{(-\alpha k^t)}.\label{Dline}\\
\intertext{{\normalsize
Suppose $\beta$ is a positive constant strictly between $0$ and $1$ such that $\beta\log_k(n)$ is an integer and $\beta \log_k(n) < \lfloor \log_k(n)\rfloor-N_{j}' - 1$. If $t \leq \beta \log_k(n)$, then $\frac{\log_k^{\circ i}(\log_k(n) -t)}{\log_k^{\circ i+1}(n)} \geq\frac{\log_k^{\circ i+1}(n^{1-\beta})}{\log_k^{\circ i+1}(n)} \geq d_i'$ for some $d_i'>0$ by Lemma~\ref{lemma:logLimit}. If $t> \beta\log_k(n)$, then $\frac{\log_k^{\circ i}(\log_k(n) -t)}{\log_k^{\circ i+1}(n)} \geq \frac{\log_k^{\circ i}(N_j'+1)}{\log_k^{\circ i+1}(n)}$. We split up the sum in $D_n$ in two parts. One sum with $t\leq \beta\log_k(n)$ and one with $t> \beta\log_k(n)$. Continuing from~(\ref{Dline}) we get}}
         &\leq d_{15}\frac{1}{n\prod\limits_{i=1}^{j}\log_k^{\circ i}(n)}\bigg( \sum_{t=1}^{\beta \log_k(n)} \frac{k^{t}}{\prod\limits_{i=0}^{j-1}d_i'} \exp{(-\alpha k^t)}+\prod\limits_{i=0}^{j-1}\bigg(\frac{\log_k^{\circ i+1}(n)}{\log_k^{\circ i}(N_{j}'+1)}\bigg)\sum_{t=\beta \log_k(n)+1}^{\lfloor \log_k(n)\rfloor-N_{j}' - 1} k^{t} \exp{(-\alpha k^t)}\bigg)\nonumber \\
         &\leq d_{15}\frac{1}{n\prod\limits_{i=1}^{j}\log_k^{\circ i}(n)}\bigg( d_{16}\sum_{t=1}^{\infty } t\exp{(-\alpha t)}+d_{17}\prod\limits_{i=1}^{j}\log_k^{\circ i}(n)\sum_{t=kn^\beta}^{\infty} t \exp{(-\alpha t)}\bigg).\nonumber
\end{align}}%
The first and second sum are both clearly convergent.  It is also easy to show that both of them can be bounded by a constant multiplied by the first term. Thus, we have that \[D_n \leq d_{15}\frac{1}{n\prod_{i=1}^{j}\log_k^{\circ i}(n)}\bigg( d_{18}+d_{19}\prod\limits_{i=1}^{j}\log_k^{\circ i}(n)\frac{kn^\beta}{\exp{(\alpha k n^\beta)}}\bigg)\leq d_{20}\frac{1}{n\prod_{i=1}^{j}\log_k^{\circ i}(n)}.\] 
Putting everything together and continuing from line~(\ref{Pline}), we get
\[ P_k(n) \leq c'k^n\bigg(d_{13}D_n+d_{14}\frac{1}{n\prod_{i=1}^{j}\log_k^{\circ i}(n)}\bigg) +  d_{12}\frac{k^n}{n^2} \leq c_j'\frac{k^n}{n\prod_{i=1}^{j}\log_k^{\circ i}(n)}\]
for some constant $c_j'>0$.
\end{proof}

\section{Open problems}
We conclude by posing some open problems.

In this paper we showed that $C_k(n) \in \Theta(\frac{k^n}{n})$. In other words, we showed that $C_k(n)$ can be bounded above and below by a constant times $k^n/n$ for $n$ sufficiently large. Can we do better than this? 

\begin{open}
Does the limit \[\lim_{n\to \infty} \frac{C_k(n)n}{k^n}\]exist? If it does exist, what does the limit evaluate to? If it does not exist, evaluate  \[\liminf_{n\to \infty} \frac{C_k(n)n}{k^n} \ \ \text{ and } \ \ \limsup_{n\to \infty} \frac{C_k(n)n}{k^n}.\]\label{open:one}
\end{open}

In this paper, we also gave a family of upper and lower bounds for $P_k(n)$. But for every $j\geq 0$, the upper and lower bounds on $P_k(n)$ are asymptotically separated by a factor of $1/\log_k^{\circ j}(n)$. Let $\ell_n$ denote the smallest positive integer such that $\log_k^{\circ \ell+1}(n) \leq 1$. Let $\log_k^*(n)$ denote the product \[\prod_{j= 1}^{\ell_n} \log_k^{\circ j}(n).\] 

\begin{open}
    Is $P_k(n)\in \Theta(\frac{k^n}{n\log_k^*(n)})$?\label{open:two}
\end{open}
This problem can probably be solved by a careful analysis of the constants introduced on every step in Section~\ref{section:privileged}.

\begin{open}
     Does the limit \[\lim_{n\to \infty} \frac{P_k(n)n\log_k^*(n)}{k^n}\] exist? If it does, what does the limit evaluate to? If it does not exist, evaluate  \[\liminf_{n\to \infty} \frac{P_k(n)n\log_k^*(n)}{k^n} \ \ \text{ and } \ \ \limsup_{n\to \infty} \frac{P_k(n)n\log_k^*(n)}{k^n}.\]\label{open:three}
\end{open}

We suspect that the first limit in problem~\ref{open:one} and the first limit in problem~\ref{open:three} do not exist due to a result of Guibas and Odlyzko~\cite{Guibas&Odlyzko:1978} on prefix-synchronized codes. Every codeword in a prefix-synchronized code of length $n$ begins with the same prefix $u$ of length $p<n$. Each codeword is a prefix of a closed word of length $n+p$ that is closed by  $u$. They proved that, for $2\leq k \leq 4$, the size $M_n$ of a maximal prefix-synchronized code of length $n$ oscillates such that the limit $\lim_{n\to\infty}M_nn/k^n$ does not exist. They mention that their approach can be generalized for $k\geq 5$, but that the proof is much more complicated.

\section{Acknowledgements}
Thanks to Jeffrey Shallit for introducing me to this problem and for helpful discussions and suggestions.

\bibliographystyle{new2}

\bibliography{abbrevs,simplest}

\end{document}